\documentclass[11pt]{amsart}
\usepackage{amsmath,amssymb, amsthm}
\usepackage{pdfsync}
\usepackage{color}
\usepackage{comment}
\usepackage{verbatim}
\usepackage{multirow}
\usepackage{tabularx}
\usepackage{setspace}

\newcommand\Z{{\mathbb Z}}

\newcommand\inv{^{-1}}

\newcommand{\CC}{\mathcal C}
\newcommand{\GG}{\mathcal G}

\newcommand{\ga}{\gamma}

\newtheorem{theorem}{Theorem}[section]

\newtheorem{lemma}[theorem]{Lemma}

 \newtheorem{remark}[theorem]{Remark}

\begin{document}
\title[Automata groups generated by step-2 nilpotent groups]{Automata groups generated by Cayley machines of groups of nilpotency class two}
\author{Ning Yang}

\maketitle
\begin{abstract}We show presentations of automata groups generated by Cayley machines of finite groups of nilpotency class two and these automata groups are all cross-wired lamplighters.
\end{abstract}

\section{Introduction}

Automata groups gain significance since they provide interesting examples in geometric group theory and related fields, for instance, the celebrated Grigorchuk group \cite{G0,G}.

Grigorchuk and \.{Z}uk \cite{GZ} showed that the lamplighter group $\Z/2\Z\wr\Z$ can be constructed as the automata group of a 2-state automaton. Using this automaton, they computed the spectra associated to random walks on the lamplighter group and found a counterexample of the strong form of the Atiyah conjecture. The 2-state automaton is invertible and called the Cayley machine of $\Z/2\Z\wr\Z$. Cayley machine
seems to have first been considered by Krohn and Rhodes in \cite{KR1,KR2,KRT}, see also \cite{Eilenberg}. In fact, each finite group has its associated Cayley machine. Silva and Steinberg \cite{SiSt} showed that, for a finite abelian group $F$, the lamplighter group $F \wr\Z$ can be generated by the Cayley machine of $F$. 

However, for finite nonabelian groups, automata groups generated by their Cayley machines do not embed in a wreath product of a finite group with a torsion free group \cite{KSS, Yang}. Pochon \cite{Pochon} first studied the Cayley machine of a finite nonabelian group and found the automata group structure for the dihedral group of order $8$. The author \cite{Yang} gave the automata group presentations for finite step-2 nilpotent groups with central squares. 

We study the automata groups generated by Cayley machines of finite step-2 nilpotent groups. Here is our main theorem.

\begin{theorem}\label{main}
	The automata group generated by the Cayley machine of a finite group $G$ of nilpotency class two has the following presentation: $$\langle x,G  \ |\  [ x^ngx^{-n}, \ h]=\prod\limits_{j=1}^{n}x^j[g\inv,h^{a_{nj}}]x^{-j},n\in\Z^+,g,h\in G\rangle,$$ where $[g ,h]=g\inv h\inv gh$, $a_{ij}=(-1)^{i-j-1}2^{2j-i}C_{j-1}^{i-j-1}-(-1)^{i-j}2^{2j-i-1}C_{j-1}^{i-j}$ and $C_l^m$ are binomial coefficients for $0\leq m\leq l$ and $0$ otherwise.
\end{theorem}

All conjugates in the products on the right hand side of the relations are central in the automata group, since $G$ is step-2 nilpotent. The proof of the theorem is by comparing wreath product coordinates.

Cayley automata groups generalizes lamplighter groups $F \wr\Z$ if $F$ is finite abelian.  For nonabelian $F$, $F \wr\Z$ can never be automata groups, because automata groups are always residually finite \cite{GNS} but not the other \cite{Gru}. 

Lamplighter groups are always cocompact lattices in the isometry groups of Diestel-Leader graphs. Indeed, Eskin, Fisher, and Whyte \cite{EFW1, EFW2, EFW3} proved  that a finitely generated group is quasi-isometric to a
lamplighter group if and only if it acts properly and cocompactly by
isometries on some Diestel-Leader graph. Cornulier, Fisher and Kashyap \cite{CFK} call these lattices \textit{cross-wired lamplighter groups}. They also showed that cross-wired lamplighters are not necessarily lamplighters.

It is suggested by Bartholdi that automata groups associated to some finite groups might provide examples of cross-wired lamplighters. We show that the automata groups in Theorem \ref{main} are cross-wired lamplighter groups, by the algebraic characterization of cross-wired lamplighters given in \cite{CFK}.

\begin{theorem}\label{cwl}
	The automata groups generated by Cayley machines of finite step-2 nilpotent groups are cross-wired lamplighter groups.
\end{theorem}

\section{Notations}

Let $G=\{g_1=1,g_2,g_3,\cdots,g_k\}$ be a step-2 nilpotent group of order $k$. Let $\mathcal C(G)$ be the Cayley machine associated to $G$, i.e., the 
automaton with state $G$ and
alphabet $G$. Both the transition and the output functions are 
group multiplication, i.e., at state $g_i$ on input $g$ the
machine goes to state $g_ig$ and outputs $g_ig$. 

For simplicity, we write $\CC$ as $\CC (G)$.  Then $\CC$ is invertible. Its inverse $\CC\inv$ is a reset automaton with states and input alphabet
$G$, where at state $g_i$ on input $g$ the automaton goes to
state $g$ and outputs $g_i\inv g$. 

Let $\GG(\CC)$ be the automata group generated by $\CC$, then $$\GG(\CC)=<\CC_{g_1},\CC_{g_2},\cdots,\CC_{g_k}>=\GG(\CC\inv).$$ $\GG(\CC)$ acts on rooted $k$-tree or the set of right infinite words $G^{\omega}$ by automorphisms and has an embedding $\GG(\CC)\hookrightarrow G\wr \GG (\CC)$. In wreath product coordinates:
$$\CC_g = g(\CC_{gg_1},\cdots,\CC_{gg_k}),\  \CC\inv_g=g\inv(\CC\inv_{g_1},\cdots,\CC\inv_{g_k}).$$ 

Let $x=\CC\inv_{g_1}=\CC\inv_{1}.$ Since $$x\CC_g=1\inv(\CC\inv_{g_1},\cdots,\CC\inv_{g_k})g(\CC_{gg_1},\cdots,\CC_{gg_k})= g(\CC_{gg_1}\inv\CC_{gg_1},\cdots,\CC_{gg_k}\inv\CC_{gg_k})=g(1,\cdots,1),$$ we can identify $G$ with a subgroup of $\GG(\CC)$ via 
$g\mapsto x\CC_g.$ Then $\CC\inv_g=g\inv x, \CC_g=x\inv g.$

Moreover,  $x^n g x^{-n}=g\Big(\CC^{-n}_{gg_1}\CC^{n}_{g_1}, \cdots,\CC^{-n}_{gg_k}\CC^{n}_{g_k}\Big)$ in wreath product coordinates. 

Let
$$N:=\langle x^nGx^{-n}\mid n\in \Z\rangle.$$
It is shown in \cite{SiSt} that $x$ has infinite order, $N$ is a
locally finite group and $$\GG(\CC) = N\rtimes\langle x\rangle,$$
where $x$ acts on $N$ by conjugates.

The \emph{depth} of an element
$\ga\in \GG(\CC)$ \cite{SiSt} is the least integer $n$ (if it exists, otherwise infinity) so that
$\ga$ only changes the first $n$ letters of a word in $G^{\omega}$. $x^ngx^{-n}$ has depth  $n+1$ for $n\geq 0$ and infinity for $n<0$.

For more details on Cayley machines and automata groups, please see \cite{BGS,Rhodes,SiSt,GZ,KSS}.

To prove the main theorem, we need following lemmas to determine the powers $a_{ij}$ in commutators.

We now define $a_{ij}$ recursively. Let $a_{ij}\in \Z,i,j\in \Z^+.$ 

Suppose  $$a_{11}=-1,a_{21}=1,$$ $$a_{1j}=0, \ \ j>1,$$ $$a_{i1}=0,\ \ i>2,$$ and $$a_{ij}=\sum\limits_{l=1}^{i-1}a_{l,j-1}-\sum\limits_{l=1}^{i-1}a_{l,j}+a_{i-1,j-1},\ \  i>1,j>1.$$
Easily we see $a_{ij}$ are well defined uniquely.

If we consider $(a_{ij})$ as an infinite matrix, then the following table gives the $11\times 11$ upper left block in the matrix. We leave it blank if the element is $0$. For example, $a_{53}=-5,a_{35}=0$. 

\begin{table}[h]
	\begin{center}
		\begin{spacing}{1.36}
			\caption{ Block $(a_{ij}),1\leq i,j\leq 11$:}

			\begin{tabular}{|c|c|c|c|c|c|c|c|c|c|c|}		
				\hline
				-1 &&&&&&&&&& \\
				\hline
				1& -2 &&&&&&&&&\\
				\hline
				&3 &-4 &  &  &&&&&& \\
				\hline
				& -1 & 8 & -8 &&&&&&&\\
				\hline
				&  & -5 & 20 & -16& &&&&&\\
				\hline
				&  & 1 & -18 & 48 &-32&&&&&\\
				\hline
				&  & & 7 & -56 &112&-64&&&&\\
				\hline
				&  & & -1 & 32 &-160&256&-128&&&\\
				\hline
				&  & &  & -9 &120&-432&576&-256&&\\
				\hline
				&  & &  & 1 &-50&400&-1120&1280&-512&\\
				\hline
				&  & &  &  &11&-220&1232&-2816&2816&-1024\\
				\hline
			\end{tabular}
		\end{spacing}
	\end{center}
\end{table}

In particular, the automata group in Theorem \ref{main} has one relation as:
$$[ x^7gx^{-7}, \ h]=x^4[g\inv,h^7]x^{-4}x^5[g\inv,h^{-56}]x^{-5}x^6[g\inv,h^{112}]x^{-6}x^7[g\inv,h^{-64}]x^{-7}.$$

\begin{lemma}\label{coeff}
	 For $i>1,j>1$, 	
	\begin{equation}\label{coeffi}
	a_{ij}=(-1)^{i-j-1}2^{2j-i}C_{j-1}^{i-j-1}-(-1)^{i-j}2^{2j-i-1}C_{j-1}^{i-j},
	\end{equation} 
	where $C_n^m$ are binomial coefficients for $0\leq m\leq n$ and $0$ otherwise.
\end{lemma}

\begin{proof}
	By assumption and direct computation, we have $a_{22}=a_{11}-a_{12}+a_{11}=-2$ and $a_{2j}=0$ for $j>2$.
	For $i>2,j>1$, we compute
	\begin{align*}
	a_{ij} &=\sum\limits_{l=1}^{i-1}a_{l,j-1}-\sum\limits_{l=1}^{i-1}a_{l,j}+a_{i-1,j-1}
	\\ &=\sum\limits_{l=1}^{i-1}a_{l,j-1}-\sum\limits_{l=1}^{i-2}a_{l,j}-(\sum\limits_{l=1}^{i-2}a_{l,j-1}-\sum\limits_{l=1}^{i-2}a_{l,j}+a_{i-2,j-1})+a_{i-1,j-1}
	\\ &=2a_{i-1,j-1}-a_{i-2,j-1}.
	\end{align*}
	Plugging (\ref{coeffi}) into the simplified recursive formula above, we verify 
	 \begin{align*}
	 2a_{i-1,j-1}-a_{i-2,j-1} &=2\big((-1)^{i-j-1}2^{2j-i-1}C_{j-2}^{i-j-1}-(-1)^{i-j}2^{2j-i-2}C_{j-2}^{i-j}\big)\\
	 &\ \ \ -\big((-1)^{i-j-2}2^{2j-i}C_{j-2}^{i-j-2}-(-1)^{i-j-1}2^{2j-i-1}C_{j-2}^{i-j-1}\big)
	 \\ &=\big((-1)^{i-j-1}2^{2j-i}C_{j-2}^{i-j-1}-(-1)^{i-j-2}2^{2j-i}C_{j-2}^{i-j-2}\big)\\
	 &\ \ \ -\big((-1)^{i-j}2^{2j-i-1}C_{j-2}^{i-j}-(-1)^{i-j-1}2^{2j-i-1}C_{j-2}^{i-j-1}\big)
	 \\ &=(-1)^{i-j-1}2^{2j-i}C_{j-1}^{i-j-1}-(-1)^{i-j}2^{2j-i-1}C_{j-1}^{i-j}\\ &=a_{i,j}.
	 \end{align*}
	 
	 It competes the proof, since $a_{ij}$ are well defined uniquely.
\end{proof}

\begin{lemma}\label{neg1}
	For each $n\in \Z^+$, $\sum\limits_{j=1}^\infty a_{nj}=-1, a_{nn}=-2^{n-1}$ and $a_{nl}=0$ for $l>n$.	
\end{lemma}
The lemma can be easily proved by induction.

\begin{lemma}\label{aij}
	For each $m\leq n\in \Z^+$, $\sum\limits_{j=m}^{n}a_{nj}+\sum\limits_{j=m}^{n}a_{jm}=\sum\limits_{j=m+1}^{n+1}a_{n+1,j}$.
\end{lemma}
\begin{proof}
	\begin{align*}
	\sum\limits_{j=m+1}^{n+1}a_{n+1,j}&= \sum\limits_{j=m+1}^{n+1}\big(\sum\limits_{l=1}^{n}a_{l,j-1}-\sum\limits_{l=1}^{n}a_{l,j}+a_{n,j-1}\big)\\
	&= \sum\limits_{l=1}^na_{lm}+\sum\limits_{j=m+2}^{n+1}\sum\limits_{l=1}^{n}a_{l,j-1}-\sum\limits_{j=m+1}^{n}\sum\limits_{l=1}^{n}a_{l,j}+\sum\limits_{j=m+1}^{n+1}a_{n,j-1}\\
	&=\sum\limits_{l=1}^m0+\sum\limits_{l=m}^na_{lm}+\sum\limits_{j=m+1}^{n+1}a_{n,j-1}=\sum\limits_{j=m}^{n}a_{jm}+\sum\limits_{j=m}^{n}a_{nj}.
	\end{align*}
\end{proof}

\section{proof of the main theorem}

The following facts about step-2 nilpotent groups are straightforward and heavily used in the following calculations. 
\begin{lemma}
 $[g,h][h,g]=1,[g,h^m][g,h^n]=[g,h^{m+n}],g,h\in G,m,n\in \Z.$
\end{lemma}	

We show the relations first.

\begin{lemma}\label{relation}
$[ x^ngx^{-n}, \ h]=\prod\limits_{j=1}^{n}x^j[g\inv,h^{a_{nj}}]x^{-j},n\in\Z^+,g,h\in G.$
\end{lemma}

\begin{proof} 
	The proof is by induction on n.
	
	The case of $n=1$: It suffices to show $[xgx\inv, h]=x[g\inv,h\inv]x\inv.$ Indeed, in wreath product coordinates, we have
	\begin{align*}
	x g x\inv h &= g\Big(\CC\inv_{gg_1}\CC_{g_1},\  \CC\inv_{gg_2}\CC_{g_2},\ \cdots,\  \CC\inv_{gg_k}\CC_{g_k}\Big)h(1,\cdots,1)
	\\ &= gh\Big(\CC\inv_{ghg_1}\CC_{hg_1},\ \cdots,\  \CC\inv_{ghg_k}\CC_{hg_k}\Big)\\
	 &=gh\Big(g_1\inv h\inv g\inv hg_1,\cdots,g_k\inv h\inv g\inv hg_k\Big)\\
	 &=gh\Big(\CC\inv_{h\inv ghg_1}\CC_{g_1},\cdots,\  \CC\inv_{h\inv ghg_k}\CC_{g_k}\Big)\\
	 &=hxh\inv ghx\inv=h x g x\inv x[g\inv,h\inv]x\inv.
	\end{align*}
	
	Inductive hypothesis: we assume the lemma is true for all $n\leq m$. 
	
	In particular we have 
	$$ x^mgx^{-m}h=hx^mgx^{-m}\prod\limits_{j=1}^{m}x^j[g\inv,h^{a_{mj}}]x^{-j}.$$

	In wreath product coordinates,  $$x^mgx^{-m}h=gh\Big(\CC_{ghg_1}^{-m}\CC_{hg_1}^m,\cdots,\CC_{ghg_k}^{-m}\CC_{hg_k}^m\Big),$$
	
	and, applying Lemma \ref{neg1},
	
	\begin{align*}
	&hx^mgx^{-m}\prod\limits_{j=1}^{m}x^j[g\inv,h^{a_{mj}}]x^{-j}=h(\prod\limits_{j=1}^{m}x^j[g\inv,h^{a_{mj}}]x^{-j})x^mgx^{-m}\\
	&=h\prod\limits_{j=1}^{m}[g\inv,h^{a_{mj}}]g\Big(\CC\inv_{\prod_{j=1}^{m}[g\inv,h^{a_{mj}}]gg_1}\CC^1_{\prod_{j=2}^{m}[g\inv,h^{a_{mj}}]gg_1}\CC^{-2}_{\prod_{j=2}^{m}[g\inv,h^{a_{mj}}]gg_1}\CC^2_{\prod_{j=3}^{m}[g\inv,h^{a_{mj}}]gg_1}\\
	&\ \ \ \ \cdots\CC^{-m}_{\prod_{j=m}^{m}[g\inv,h^{a_{mj}}]gg_1}\CC^m_{g_1},\cdots\Big)\\
	&=gh\Big(\CC\inv_{\prod_{j=1}^{m}[g\inv,h^{a_{mj}}]gg_1}\CC^{-1}_{\prod_{j=2}^{m}[g\inv,h^{a_{mj}}]gg_1}\CC^{-1}_{\prod_{j=3}^{m}[g\inv,h^{a_{mj}}]gg_1} \cdots\CC^{-1}_{\prod_{j=m}^{m}[g\inv,h^{a_{mj}}]gg_1}\CC^m_{g_1},\cdots\Big).
	\end{align*}
	
	Then for any $f\in G$, we have 
	$$\CC_{ghf}^{-m}\CC_{hf}^m=\CC\inv_{\prod_{j=1}^{m}[g\inv,h^{a_{mj}}]gf}\CC^{-1}_{\prod_{j=2}^{m}[g\inv,h^{a_{mj}}]gf} \cdots\CC^{-1}_{\prod_{j=m}^{m}[g\inv,h^{a_{mj}}]gf}\CC^m_{f}.$$
	
	The case of $n=m+1$: we consider wreath product coordinates.
	
	\begin{align*}
	&\CC_{ghf}^{-m-1}\CC_{hf}^{m+1}=\CC\inv_{ghf}\CC_{ghf}^{-m}\CC_{hf}^{m}\CC_{hf}^{1}\\
	&=\CC\inv_{ghf}\CC\inv_{\prod_{j=1}^{m}[g\inv,h^{a_{mj}}]gf}\CC\inv_{\prod_{j=2}^{m}[g\inv,h^{a_{mj}}]gf}\cdots\CC^{-1}_{\prod_{j=m}^{m}[g\inv,h^{a_{mj}}]gf}\CC^m_{f}\CC_{hf}^{1}\\
	&= (ghf)^{\inv} x f\inv g\inv \prod_{j=1}^{m}[g,h^{a_{mj}}]x f\inv g\inv \prod_{j=2}^{m}[g,h^{a_{mj}}]x  \cdots f\inv g\inv [g,h^{a_{mm}}]x(x\inv f)^m x\inv hf\\
	&=(ghf)^{-1}x f\inv g\inv \prod_{j=1}^{m}[g,h^{a_{mi}}]x\inv x^2 f\inv g\inv \prod_{j=2}^{m}[g,h^{a_{mj}}] x^{-2} \cdots x^mf\inv g\inv [g,h^{a_{mm}}]x^{-m}\\ 
	&\ \ \ \  x^m fx^{-m}\cdots x fx\inv  hf\\
	&=(ghf)^{-1}hx f\inv g\inv \prod_{j=1}^{m}[g,h^{a_{mj}}]x\inv x^2 f\inv g\inv \prod_{j=2}^{m}[g,h^{a_{mj}}] x^{-2} \cdots x^mf\inv g\inv [g,h^{a_{mm}}]x^{-m}\\ 
	&\ \ \ \  x^m fx^{-m}\cdots x fx\inv  f\cdot \prod\limits_{s=1}^{1}x^s[gf,h^{a_{1s}}]x^{-s} \prod\limits_{s=1}^{2}x^s[gf,h^{a_{2s}}]x^{-s}\cdots \prod\limits_{s=1}^{m}x^s[gf,h^{a_{ms}}]x^{-s}\\
	&\ \ \ \ \prod\limits_{s=1}^{m}x^s[f\inv,h^{a_{ms}}]x^{-s}\cdots \prod\limits_{s=1}^{1}x^s[f\inv,h^{a_{1s}}]x^{-s}\\
	&=(ghf)^{-1}hx f\inv g\inv \prod_{j=1}^{m}[g,h^{a_{mj}}]x\inv x^2 f\inv g\inv \prod_{j=2}^{m}[g,h^{a_{mj}}] x^{-2} \cdots x^mf\inv g\inv [g,h^{a_{mm}}]x^{-m}\\ 
	&\ \ \ \  x^m fx^{-m}\cdots x fx\inv  f\cdot \prod\limits_{s=1}^{1}x^s[g,h^{a_{1s}}]x^{-s} \prod\limits_{s=1}^{2}x^s[g,h^{a_{2s}}]x^{-s}\cdots \prod\limits_{s=1}^{m}x^s[g,h^{a_{ms}}]x^{-s}\\
	&=(ghf)^{-1}hx f\inv g\inv \prod_{j=1}^{m}[g,h^{a_{mj}}]\prod_{s=1}^{m}[g,h^{a_{s1}}]x\inv x^2 f\inv g\inv \prod_{j=2}^{m}[g,h^{a_{mj}}]\prod_{s=2}^{m}[g,h^{a_{s2}}] x^{-2} \cdots\\
	&\ \ \ \ x^mf\inv g\inv \prod_{j=m}^{m}[g,h^{a_{mj}}]\prod_{s=m}^{m}[g,h^{a_{sm}}]x^{-m} x^m fx^{-m}\cdots x fx\inv  f\\
	&=(gf)^{\inv}[g,h\inv]x(gf)^{\inv}\prod_{j=2}^{m+1}[g,h^{a_{m+1,j}}]x\inv x^2(gf)^{\inv}\prod_{j=3}^{m+1}[g,h^{a_{m+1,j}}]x^{-2}\cdots\\
	&\ \ \ \ x^m(gf)\inv \prod_{j=m+1}^{m+1}[g\inv,h^{a_{m+1,j}}]x^{-m} x^m fx^{-m}\cdots x fx\inv f\\
	&=f\inv g\inv \prod_{j=1}^{m+1}[g,h^{a_{m+1,j}}]x f\inv g\inv \prod_{j=2}^{m+1}[g,h^{a_{m+1,j}}]x  \cdots f\inv g\inv \prod_{j=m+1}^{m+1}[g,h^{a_{m+1,j}}]x(x\inv f)^{m+1}\\
	&=\CC\inv_{\prod_{j=1}^{m+1}[g\inv,h^{a_{m+1,j}}]gf}\CC\inv_{\prod_{j=2}^{m+1}[g\inv,h^{a_{m+1,j}}]gf}\cdots\CC^{-1}_{\prod_{j=m+1}^{m+1}[g\inv,h^{a_{m+1,j}}]gf}\CC^{m+1}_{f},
	\end{align*}
	
	where the second to the last equality is by Lemma \ref{neg1}, the third to the last equality is by Lemma \ref{aij}, and we commute the last $h$ to the far left in the fifth equality.
	
Therefore, 	$ x^{m+1}gx^{-m-1}h=hx^{m+1}gx^{-m-1}\prod\limits_{j=1}^{m+1}x^j[g\inv,h^{a_{m+1,i}}]x^{-j}$, and hence the case of $m+1$ is true. The lemma follows.
	
\end{proof}

\begin{lemma}\label{form}
	Given any non-trivial torsion element $\gamma\in\GG(\CC)$, it can be written uniquely in the form:
	\begin{equation}\label{order}
	\gamma= x^{i_1}f_1x^{-i_1}x^{i_2}f_2x^{-i_2}\dots x^{i_j}f_{j}x^{-i_j},
	\end{equation}
	where $i_1<i_2<\ldots <i_j$ are integers, $j\in \Z^+$ , and $f_1,f_2,\ldots,f_j\in G\setminus{\{g_1=1\}}$.
\end{lemma}

\begin{proof}
	Since $\gamma$ is a product of conjugates by $x$ of elements in $G$, by Lemma \ref{relation}, $\gamma$ can be written in the form of \eqref{order}. So it suffices to show the uniqueness.
	
	Let $\gamma= x^{i_1}f_1x^{-i_1}x^{i_2}f_2x^{-i_2}\dots x^{i_j}f_{j}x^{-i_j}=1$, where $ i_1<i_2<\ldots <i_j$ and $f_1,f_2,\ldots,f_j\in G$. Then $1=x^{-i_1}\gamma x^{i_1}=f_1x^{i_2-i_1}f_2x^{-i_2+i_1}\dots x^{i_j-i_1}f_{j}x^{-i_j+i_1}$ has depth $i_j-i_1+1\geq 1$. Hence $f_i=1, 1\leq i\leq j$.
	It completes the proof.
\end{proof}

\begin{proof}[Proof of Theorem \ref{main} ]
	Note that $\GG(\CC)=N\rtimes\langle x\rangle$. Then the theorem follows from Lemma \ref{coeff}, \ref{relation} and \ref{form}.
\end{proof}

\begin{remark}
	Theorem 1.1 in \cite{Yang} is a corollary of our main theorem. One can easily check that for each $n$, there is one and only one odd element $a_{n,[\frac{n}{2}]}$, where $[\frac{n}{2}]$ is the least integer greater than or equal to $\frac{n}{2}.$
\end{remark}

\begin{proof}[Proof of Theorem \ref{cwl}]
	The proof is almost the same as the proof of theorem 1.4 in \cite{Yang}.
\end{proof}

\medskip
\noindent
Ning Yang\\
Department of Mathematics and Science\\
Nantong Normal College\\
Nantong, Jiangsu 226011, China\\
E-mail: ntningyang@126.com

\end{document}